\documentclass[11pt]{article}
\usepackage[a4paper,margin=30mm]{geometry}
\usepackage{amsmath,amssymb,amsthm,mathtools}
\usepackage[numbers]{natbib}
\usepackage[colorlinks=true,citecolor=blue,linkcolor=blue,urlcolor=blue]{hyperref}
\usepackage{microtype}

\newtheorem{theorem}{Theorem}[section]
\newtheorem{lemma}[theorem]{Lemma}

\newtheorem{corollary}[theorem]{Corollary}
\theoremstyle{definition}

\theoremstyle{remark}
\newtheorem{remark}[theorem]{Remark}
\numberwithin{equation}{section}
\DeclareMathOperator{\Ann}{Ann}
\DeclareMathOperator{\gr}{gr}
\DeclareMathOperator{\tr}{tr}

\title{Jordan Blocks, Differential Traces, and Colon-Ideal
Noncontainment over Noetherian Quotients}
\author{Yizhi Zhang}
\date{}
\hypersetup{
  pdftitle={Jordan Blocks, Differential Traces, and Colon-Ideal Noncontainment over Noetherian Quotients},
  pdfauthor={Yizhi Zhang}
}

\begin{document}
\maketitle

\begin{abstract}
Let \(k\) be a field of characteristic zero, let \(R\) be a commutative
\(k\)-algebra, let \(I\subsetneq R\), and assume that \(A=R/I\) is
Noetherian. We prove that if \(a=[f]\in A\) is nilpotent and
\(d_{A/k}a=0\), then \(I:f\not\subseteq(I,f)\); if \(R\) is local with
maximal ideal \(\mathfrak n\), then
\(I:f\not\subseteq(I,f)+\mathfrak n(I:f)\). As applications, we prove
the isolated hypersurface case in arbitrary dimension and extend the
conclusion to non-isolated critical loci on smooth varieties. Finally,
we exhibit a five-variable isolated singularity satisfying
\(J_f:f\subseteq\overline{J_f}\), thereby disproving the
integral-closure strengthening.
\end{abstract}

\section{Introduction}\label{sec:introduction}

Let \(f\) define an isolated hypersurface singularity over a field of
characteristic zero, and let \(J_f\) be its Jacobian ideal. In the
quasihomogeneous case the Euler identity gives \(f\in J_f\), with
consequences for the Milnor and Tjurina numbers and for Jacobian
syzygies. In the general case, Hassanzadeh, Nasrollah Nejad, and Simis
asked whether
\begin{equation}\label{eq:intro-original}
J_f:f\not\subseteq(J_f,f)
\end{equation}
always holds. They proved the plane-curve case and stated the
higher-dimensional case as Conjecture~3.6
\cite{HassanzadehNasrollahNejadSimis2025}.

We turn \eqref{eq:intro-original} into a structural question about
nilpotent multiplication. If
\(A=R/J_f\) and \(a=[f]\), then \(J_f:f\) modulo \(J_f\) is
\(\Ann_A(a)\), whereas \((J_f,f)\) modulo \(J_f\) is \(aA\). Thus the
problem is equivalent to proving \(\Ann_A(a)\not\subseteq aA\), or,
equivalently, to the existence of a size-one Jordan block for \(m_a\).

Our first main theorem is more general than the original
singularity-theoretic statement.
\begin{theorem}\label{thm:intro-main}
Let \(k\) be a field of characteristic zero, let \(R\) be a commutative
\(k\)-algebra, let \(I\subsetneq R\), and assume that \(A=R/I\) is
Noetherian. For \(f\in R\), write \(a=[f]\in A\). If \(a\) is
nilpotent and \(d_{A/k}a=0\), then
\[
\boxed{I:f\not\subseteq(I,f).}
\]
\end{theorem}

In the local case we prove a stronger statement.
\begin{theorem}\label{thm:intro-local}
Under the assumptions of Theorem~\ref{thm:intro-main}, suppose in
addition that \((R,\mathfrak n)\) is local. Then
\[
\boxed{I:f\not\subseteq(I,f)+\mathfrak n(I:f).}
\]
\end{theorem}
The second theorem produces a minimal generator direction of \(I:f\)
that is not absorbed by the direction of \(f\).

The abstract theorem proves \eqref{eq:intro-original} in every
dimension, with the stronger local conclusion. It also applies to
polynomial local rings, convergent and algebraic power-series rings,
and equicharacteristic complete regular local rings. Moreover, after
subtracting the critical value, it applies at a rational critical point
of a smooth affine variety even when the critical locus is
positive-dimensional.

There is, however, a sharp limitation. The natural strengthening
\[
J_f:f\not\subseteq\overline{J_f}
\]
is false. We exhibit the five-variable example
\[
f=\frac14\sum_{i=1}^5x_i^4+(x_1+\cdots+x_5)^5
\]
for which \(\overline{J_f}=\mathfrak m^3\) and
\(J_f:f\subseteq\mathfrak m^3\). This explains why the controlled term
\(\mathfrak n(I:f)\), rather than an unrestricted larger ideal, occurs
in the local theorem.

The paper is organized as follows. Section~\ref{sec:background}
describes the original problem and its Jordan reformulation.
Section~\ref{sec:finite-trace} proves the finite-dimensional
differential-trace lemma. Section~\ref{sec:noether-rigidity} establishes
Noetherian conormal rigidity. Section~\ref{sec:colon-theorems} proves
the two main theorems. Section~\ref{sec:applications} treats Jacobian
ideals and critical loci. Section~\ref{sec:integral-counterexample}
constructs the integral-closure counterexample.

\section{Background}\label{sec:background}

Let \(k\) be a field of characteristic zero, let
\[
R=k[[x_1,\ldots,x_n]],\qquad \mathfrak m=(x_1,\ldots,x_n),
\]
and suppose that \(f\in\mathfrak m\) defines an isolated hypersurface
singularity at the origin. Its Jacobian and Tjurina ideals are
\[
J_f=(\partial_1f,\ldots,\partial_nf),\qquad I_f=(J_f,f).
\]
The dimensions of \(R/J_f\) and \(R/I_f\) are the Milnor number
\(\mu_f\) and the Tjurina number \(\tau_f\), respectively. One always has
\(\tau_f\leq\mu_f\), and
\[
\mu_f-\tau_f=\dim_k [f](R/J_f).
\]
Thus \(\mu_f=\tau_f\) if and only if \([f]=0\) in \(R/J_f\), equivalently,
\(f\in J_f\). For an isolated complex analytic hypersurface germ,
Saito's classical criterion states that \(f\in J_f\) if and only if an
analytic change of coordinates makes \(f\) weighted homogeneous with
positive weights \cite{Saito1971}. In such coordinates there are
positive rational numbers \(w_1,\ldots,w_n\) and \(d>0\) for which the
weighted Euler identity
\[
d f=\sum_{i=1}^n w_i x_i\partial_i f
\]
holds. Saito's criterion therefore unifies the numerical equality
\(\mu_f=\tau_f\), the ideal-membership condition \(f\in J_f\), and
quasihomogeneity. When \(\mu_f>\tau_f\), the difference is precisely the
dimension of the principal ideal \([f](R/J_f)\) in the Milnor algebra.
Known comparisons include \(\mu_f\leq n\tau_f\) \cite{Liu2018} and, for
plane curves, \(\mu_f<2\tau_f\) \cite{DimcaGreuel2018}.

Hassanzadeh, Nasrollah Nejad, and Simis studied quasihomogeneous isolated
singularities through syzygies and foliations and formulated
\begin{equation}\label{eq:hns-question}
J_f:f\not\subseteq(J_f,f).
\end{equation}
They established the plane-curve case
\cite[Conjecture~3.6 and Proposition~3.7]{HassanzadehNasrollahNejadSimis2025}.
The problem asks for an actual direction in \(J_f:f\) that cannot be
absorbed by \(J_f\) together with \(f\), rather than merely a comparison
of two lengths.

\section{Exact zero divisors and the finite-dimensional trace}
\label{sec:finite-trace}

Return to the isolated-singularity setting of Section~\ref{sec:background}
and put
\[
A=R/J_f,\qquad a=[f]\in A.
\]
The colon ideal
\(J_f:f=\{g\in R:gf\in J_f\}\) and the annihilator are related by
\begin{equation}\label{eq:colon-dictionary-background}
(J_f:f)/J_f=\Ann_A(a),\qquad (J_f,f)/J_f=aA.
\end{equation}
Consequently, \eqref{eq:hns-question} is equivalent to
\(\Ann_A(a)\not\subseteq aA\). Since \(A\) is Artinian local, \(a\) is
nilpotent. Inspecting the Jordan blocks of the multiplication operator
\(m_a\) gives
\[
\Ann_A(a)\not\subseteq aA
\quad\Longleftrightarrow\quad
m_a\text{ has a Jordan block of size one}.
\]
This is the linear-algebraic form of the original problem. Jordan type
or Frobenius self-adjointness alone does not exclude the possibility
that all blocks have size at least two; the additional constraint is
the differential identity \(d_{A/k}a=0\). We start from this observation
and subsequently pass from finite-length algebras to arbitrary
Noetherian quotient rings.

For an element \(b\) of a ring \(B\), write
\(\Ann_B(b)=\{x\in B:xb=0\}\). If \(B\) is a commutative \(L\)-algebra,
write
\[
d_{B/L}:B\longrightarrow\Omega_{B/L}
\]
for the universal K\"ahler derivation. When the algebra and base field
are determined by context, we abbreviate \(d_{B/L}(x)\) to \(dx\). Thus,
in Lemma~\ref{lem:finite-trace}, \(dc\) means \(d_{E/L}(c)\), whereas
\(d\varepsilon\) means \(d_{D/L}(\varepsilon)\); in
Theorem~\ref{thm:conormal-rigidity}, \(dc\) means \(d_{C/k}(c)\). For a
matrix \(X=(x_{ij})\) with entries in \(D\), set
\(dX=(d_{D/L}x_{ij})\), taking the differential entrywise.

\begin{lemma}\label{lem:colon-exact}
Let \(A\) be a commutative ring and let \(0\ne a\in A\) be nilpotent.
If \(\Ann_A(a)\subseteq aA\), then
\[
\Ann_A(a)=a\Ann_A(a^2).
\]
Consequently, if \(C=A/(a^2)\) and \(c=a\bmod(a^2)\), then
\begin{equation}\label{eq:self-exact}
c\ne0,\qquad c^2=0,\qquad\Ann_C(c)=cC.
\end{equation}
\end{lemma}
\begin{proof}
If \(x\in\Ann_A(a)\), the assumption gives \(x=ay\), and then
\(a^2y=0\). Hence
\(\Ann_A(a)\subseteq a\Ann_A(a^2)\). Conversely, if
\(x=ay\in a\Ann_A(a^2)\), then \(a^2y=0\), so
\(ax=a^2y=0\). Thus \(x\in\Ann_A(a)\), proving the reverse inclusion.

If \(c=0\), then
\(a=a^2u\), so \(a(1-au)=0\). Since \(au\) is nilpotent, \(1-au\) is a
unit, contradicting \(a\ne0\). If \(c\bar x=0\), lift \(\bar x\) to
\(x\in A\). There is a \(y\) with \(ax=a^2y\), and hence
\[
x-ay\in\Ann_A(a)=a\Ann_A(a^2).
\]
Reduction modulo \(a^2A\) yields \(\bar x\in cC\). The reverse
inclusion can be checked directly: if \(\bar x=c\bar y\), then
\(c\bar x=c^2\bar y=0\), so \(\bar x\in\Ann_C(c)\).
\end{proof}

\begin{lemma}\label{lem:finite-trace}
Let \(L\) be a field of characteristic zero and let \(E\) be a
finite-dimensional commutative \(L\)-algebra. If \(0\ne c\in E\)
satisfies \(c^2=0\) and \(\Ann_E(c)=cE\), then
\[
d_{E/L}c\notin c\Omega_{E/L}.
\]
\end{lemma}
\begin{proof}
Let \(D=L[\varepsilon]/(\varepsilon^2)\), with
\(\varepsilon\mapsto c\). This map is injective. Multiplication by
\(c\) has square zero and
\[
\ker(m_c)=\Ann_E(c)=cE=\operatorname{im}(m_c).
\]
Thus its Jordan normal form consists only of size-two blocks. If there
are \(r>0\) blocks, then \(E\simeq D^r\) as a \(D\)-module.

Consider the regular representation
\[
\rho:E\longrightarrow\operatorname{End}_D(E)\simeq M_r(D).
\]
Define
\[
T:\Omega_{E/L}\longrightarrow\Omega_{D/L},\qquad
T(b\,dx)=\tr\bigl(\rho(b)d\rho(x)\bigr).
\]
We spell out the descent to K\"ahler differentials. Additivity in
\(b\) follows from additivity of matrix trace and of \(\rho\), while
additivity in \(x\) follows from additivity of \(\rho\) and entrywise
differentiation. If \(\lambda\in L\), then
\(\rho(\lambda)=\lambda I_r\), so
\(d\rho(\lambda)=d(\lambda I_r)=0\), and hence
\(T(b\,d\lambda)=0\). For the Leibniz relation, one has
\begin{align*}
T\bigl(b\,d(xy)\bigr)
&=\tr\!\left(\rho(b)d\bigl(\rho(x)\rho(y)\bigr)\right)\\
&=\tr\!\left(\rho(b)d\rho(x)\rho(y)\right)
  +\tr\!\left(\rho(b)\rho(x)d\rho(y)\right)\\
&=\tr\!\left(\rho(by)d\rho(x)\right)
  +\tr\!\left(\rho(bx)d\rho(y)\right)\\
&=T(by\,dx)+T(bx\,dy).
\end{align*}
Here the third equality uses commutativity of \(E\) and cyclicity of
matrix trace; it does not require the entries of \(d\rho(x)\) to
commute with those of \(\rho(y)\). Hence the formula respects the
defining presentation of \(\Omega_{E/L}\). Moreover, for
\(\delta\in D\),
\[
T(\delta b\,dx)=\delta T(b\,dx),
\]
because \(\rho(\delta)=\delta I_r\). Thus \(T\) is \(D\)-linear. Since
\(\rho(c)=\varepsilon I_r\),
\begin{equation}\label{eq:trace-c}
T(dc)=r\,d\varepsilon.
\end{equation}
If \(dc=c\omega\), then \(T(dc)=\varepsilon T(\omega)=0\), because
\[
\Omega_{D/L}=L\,d\varepsilon,\qquad
\varepsilon\Omega_{D/L}=0.
\]
This contradicts \eqref{eq:trace-c}, since \(r>0\) and
\(\operatorname{char}L=0\).
\end{proof}

\begin{remark}
Jordan normal form is used only to establish freeness over the dual
numbers. The differential condition is detected by the matrix trace of
the regular representation. No Frobenius pairing, Gorenstein
hypothesis, or complete-intersection hypothesis is required.
\end{remark}

\section{Conormal rigidity over Noetherian algebras}
\label{sec:noether-rigidity}

This section reduces the Noetherian case to the finite-dimensional one.
Localization at a minimal prime produces an Artinian local ring, and a
copy of its residue field inside the ring then places it within the
scope of Lemma~\ref{lem:finite-trace}. The following coefficient-field
lemma supplies the second step.

\begin{lemma}\label{lem:coefficient-field}
Let \(k\) be a field of characteristic zero and let \((E,\mathfrak m,L)\)
be an Artinian local \(k\)-algebra with residue field
\(L=E/\mathfrak m\). Then the residue map \(E\twoheadrightarrow L\)
admits a \(k\)-algebra section.
\end{lemma}
\begin{proof}
Choose a transcendence basis \(\mathcal T\) of \(L/k\). Then \(L\) is
algebraic over \(k(\mathcal T)\), and it is separable because the
characteristic is zero. The algebra \(k(\mathcal T)\) is a localization
of a polynomial algebra (interpreted as the filtered union over finite
subsets of \(\mathcal T\)), hence is formally smooth over \(k\).
Every finite subextension of the separable algebraic extension
\(L/k(\mathcal T)\) is finite \'etale, and therefore formally
\'etale. Passing to the filtered union shows that \(L\) is formally
smooth over \(k\).

The ideal \(\mathfrak m\) is nilpotent because \(E\) is Artinian. Apply
the lifting property for formal smoothness to the \(k\)-algebra map
\(\operatorname{id}_L:L\to E/\mathfrak m=L\). It lifts to a
\(k\)-algebra map \(s:L\to E\). The composite
\(L\xrightarrow{s}E\twoheadrightarrow L\) is the identity, so \(s\) is
the required section.
\end{proof}

Once the coefficient field has been chosen, the localized Artinian ring
can be regarded as a finite-dimensional algebra over its residue field.
We can now lift the finite-dimensional obstruction of the preceding
section to Noetherian algebras.

\begin{theorem}\label{thm:conormal-rigidity}
Let \(k\) be a field of characteristic zero and let \(C\) be a
Noetherian commutative \(k\)-algebra. If \(0\ne c\in C\) satisfies
\[
c^2=0,\qquad\Ann_C(c)=cC,
\]
then
\[
\boxed{dc\notin c\Omega_{C/k}.}
\]
In particular, \(dc\ne0\).
\end{theorem}
\begin{proof}
Suppose that \(dc=c\omega\). Choose a minimal prime
\(\mathfrak p\) of \(C\) and set \(E=C_{\mathfrak p}\). Then \(E\) is
a zero-dimensional Noetherian local ring, hence an Artinian local ring.

The element \(c/1\) is nonzero in \(E\). Otherwise \(sc=0\) for some
\(s\notin\mathfrak p\), while
\[
s\in\Ann_C(c)=cC\subseteq\mathfrak p,
\]
a contradiction. Localization preserves
\[
(c/1)^2=0,\qquad\Ann_E(c/1)=(c/1)E.
\]

Let \(L=E/\mathfrak m_E\). By Lemma~\ref{lem:coefficient-field}, the
residue map has a \(k\)-algebra section \(L\hookrightarrow E\). Since
the descending chain
\[
E\supseteq\mathfrak m_E\supseteq\cdots\supseteq
\mathfrak m_E^N=0
\]
has finite length and each quotient
\(\mathfrak m_E^i/\mathfrak m_E^{i+1}\) is a finite-dimensional
\(L\)-vector space, \(E\) is a finite-dimensional commutative
\(L\)-algebra.

Localizing \(dc=c\omega\) and applying the natural surjection
\(\Omega_{E/k}\twoheadrightarrow\Omega_{E/L}\) gives
\[
d_{E/L}c=c\omega'.
\]
This contradicts Lemma~\ref{lem:finite-trace}.
\end{proof}

\section{Colon-ideal theorems}\label{sec:colon-theorems}

Theorem~\ref{thm:conormal-rigidity} concerns the differential of a
self-exact square-zero element. We now translate it back into a colon-
ideal statement. If the desired noncontainment were to fail,
Lemma~\ref{lem:colon-exact} would produce such a square-zero element in
a quotient ring, while differential vanishing would contradict
conormal rigidity. This proves the first main theorem.

\begin{theorem}\label{thm:noetherian-colon}
Let \(k\) be a field of characteristic zero, let \(R\) be a commutative
\(k\)-algebra, let \(I\subsetneq R\), and assume that \(A=R/I\) is
Noetherian. For \(f\in R\), write \(a=[f]\in A\). If \(a\) is
nilpotent and \(d_{A/k}a=0\), then
\[
I:f\not\subseteq(I,f).
\]
\end{theorem}
\begin{proof}
If \(a=0\), then \(I:f=R\), while \((I,f)=I\ne R\). Assume that
\(a\ne0\) and, toward a contradiction, that
\(I:f\subseteq(I,f)\). Modulo \(I\), this is
\(\Ann_A(a)\subseteq aA\). Lemma~\ref{lem:colon-exact} produces, in
\(C=A/(a^2)\), an element \(c=a\bmod(a^2)\) satisfying
\[
c\ne0,\qquad c^2=0,\qquad\Ann_C(c)=cC.
\]
Functoriality of K\"ahler differentials gives \(d_{C/k}c=0\), contrary
to Theorem~\ref{thm:conormal-rigidity}.
\end{proof}

The preceding theorem produces an element of the colon ideal that is
not absorbed by \((I,f)\), but it does not yet show that the element can
be chosen as a minimal generator direction. Over a local ring,
Nakayama's lemma gives the following strengthening.

\begin{theorem}\label{thm:local-strengthening}
Under the assumptions of Theorem~\ref{thm:noetherian-colon}, suppose
that \((R,\mathfrak n)\) is local and put \(K=I:f\). Then
\[
\boxed{K\not\subseteq(I,f)+\mathfrak nK.}
\]
\end{theorem}
\begin{proof}
Put
\[
M=K/I=\Ann_A(a),\qquad E=(M+aA)/aA,
\]
and let \(\mathfrak m=\mathfrak n/I\). Theorem
\ref{thm:noetherian-colon} says that \(E\ne0\). If
\[
K\subseteq(I,f)+\mathfrak nK,
\]
then reduction modulo \(I\) gives \(M\subseteq aA+\mathfrak mM\).
The reverse containment is not asserted; after passing to
\((M+aA)/aA\), however, the displayed inclusion becomes
\[
E\subseteq\mathfrak mE.
\]
The reverse inclusion \(\mathfrak mE\subseteq E\) is automatic, hence
\(E=\mathfrak mE\). The module \(M=\Ann_A(a)\) is an ideal of the
Noetherian ring \(A\), so \(M\), and therefore its quotient \(E\), is
finitely generated. Nakayama's lemma would imply \(E=0\), a
contradiction.
\end{proof}

\begin{remark}\label{rem:minimal-generator}
The local theorem provides \(g\in I:f\) such that
\[
g\notin(I,f)+\mathfrak n(I:f).
\]
Thus \(g\) determines a genuine minimal generator direction of \(I:f\)
that is not absorbed by the direction of \(f\).
\end{remark}

\section{Jacobian ideals and critical loci}\label{sec:applications}

We now return to the original Jacobian-ideal problem from
Section~\ref{sec:background}. To apply
Theorem~\ref{thm:local-strengthening} to
\[
A=k[[x_1,\ldots,x_n]]/J_f,\qquad a=[f],
\]
we must verify that \(A\) is Noetherian, that \(a\) is nilpotent, and
that \(d_{A/k}a=0\). When \(J_f\) is \(\mathfrak m\)-primary, the first
two statements follow because \(A\) is Artinian local. Differential
vanishing requires a separate argument.

One cannot, without further justification, write the differential of a
formal power series as \(df=\sum_i(\partial_i f)\,dx_i\). We use ordinary
K\"ahler differentials, and ordinary derivations need not be continuous
for the \(\mathfrak m\)-adic topology, so the infinite-series formula is
not directly available. Instead, modulo \(J_f\), we replace \(f\) by a
finite polynomial truncation, for which the usual finite
differentiation formula applies. The next lemma carries this out.

\begin{lemma}\label{lem:finite-jet}
Let \(R=k[[x_1,\ldots,x_n]]\), assume that \(J_f\) is
\(\mathfrak m\)-primary, and set \(A=R/J_f\) and \(a=[f]\in A\).
Then \(d_{A/k}a=0\).
\end{lemma}
\begin{proof}
Choose \(N\) such that \(\mathfrak m^N\subseteq J_f\), and truncate
\(f\) at total degree \(N\), obtaining a polynomial \(P\). Then
\[
f-P\in\mathfrak m^{N+1}\subseteq J_f,\qquad
\partial_i f-\partial_iP\in\mathfrak m^N\subseteq J_f.
\]
Since \(\partial_i f\in J_f\), every \(\partial_iP\) vanishes in \(A\).
The elements \(P\) and \(f\) represent the same class in \(A\), so
\[
d_{A/k}a=d_{A/k}[P]
=\sum_i[\partial_iP]\,d[x_i]=0.
\]
\end{proof}

Thus the abstract local colon theorem applies to the original formal
power-series setting and gives the required conclusion.

\begin{corollary}\label{cor:hns}
Let \(k\) be a field of characteristic zero and let
\(R=k[[x_1,\ldots,x_n]]\). If \(J_f\) is \(\mathfrak m\)-primary, then
\[
\boxed{J_f:f\not\subseteq
(J_f,f)+\mathfrak m(J_f:f).}
\]
In particular, \(J_f:f\not\subseteq(J_f,f)\).
\end{corollary}
\begin{proof}
The algebra \(A=R/J_f\) is Artinian local, so \([f]\) is nilpotent.
Apply Lemma~\ref{lem:finite-jet} and
Theorem~\ref{thm:local-strengthening}.
\end{proof}

The formal power-series case uses a finite truncation to avoid the
continuity issue. For a smooth finite-type algebra, the cotangent module
is finite projective, and a dual basis expresses \(df\) directly as a
linear combination of the derivatives \(\delta(f)\). This gives the
corresponding local statement.

\begin{corollary}\label{cor:smooth-local}
Let \((R,\mathfrak n)\) be a localization of a smooth finitely generated
\(k\)-algebra, where \(\operatorname{char}k=0\), and define
\[
J_f=(\delta(f):\delta\in\operatorname{Der}_k(R)).
\]
If \(R/J_f\) is Artinian local and \(f\in\mathfrak n\), then
\[
J_f:f\not\subseteq(J_f,f)+\mathfrak n(J_f:f).
\]
\end{corollary}
\begin{proof}
Smoothness makes \(\Omega_{R/k}\) finite projective. Choose a dual
basis \(\omega_1,\ldots,\omega_s\in\Omega_{R/k}\) and
\(\varphi_1,\ldots,\varphi_s\in
\operatorname{Hom}_R(\Omega_{R/k},R)\), so that
\[
\omega=\sum_j\varphi_j(\omega)\omega_j
\quad\text{for every }\omega\in\Omega_{R/k}.
\]
By the universal property of K\"ahler differentials, every
\(\varphi_j\) corresponds to a derivation
\(\delta_j\in\operatorname{Der}_k(R)\). Applying the dual-basis
identity to \(df\) gives
\[
df=\sum_j\delta_j(f)\omega_j\in J_f\Omega_{R/k},
\]
hence
\(d_{R/J_f}[f]=0\). The remaining assumptions imply that \([f]\) is
nilpotent, so Theorem~\ref{thm:local-strengthening} applies.
\end{proof}

The same proof applies to polynomial local rings, convergent real or
complex power-series rings, algebraic power-series rings, and
equicharacteristic complete regular local rings containing a
characteristic-zero coefficient field. The formal case uses
Lemma~\ref{lem:finite-jet}; the analytic case uses the corresponding
finite-dimensional critical quotient.

The preceding two corollaries assume that the critical quotient is
Artinian, so nilpotence of \([f]\) follows from finite length. When the
critical locus is positive-dimensional, this implication is no longer
available. We first record a fact about differential constants in
function fields.

\begin{lemma}\label{lem:differential-constants}
Let \(K/k\) be a finitely generated field extension of characteristic
zero. If \(u\in K\) satisfies \(d_{K/k}u=0\), then \(u\) is algebraic
over \(k\).
\end{lemma}
\begin{proof}
Choose a separating transcendence basis
\(t_1,\ldots,t_d\) for \(K/k\). Then \(K\) is finite separable over
\(k(t_1,\ldots,t_d)\), and
\[
\Omega_{K/k}\simeq
\bigoplus_{i=1}^dK\,dt_i.
\]
If \(u\) were transcendental over \(k\), it could be included in a
transcendence basis after replacing the chosen basis. In the
corresponding decomposition, \(du\) would be a basis vector and could
not vanish. Thus \(u\) is algebraic over \(k\).
\end{proof}

The lemma allows us to work component by component. Differential
vanishing makes the function algebraic over the base field on each
irreducible component; after subtracting its value at the chosen
rational critical point, that algebraic constant must be zero. This
recovers the nilpotence needed for the main theorem.

\begin{corollary}\label{cor:positive-dimensional}
Let \(R\) be a smooth finitely generated algebra over a field \(k\) of
characteristic zero, let \(f\in R\), and let \(x\) be a \(k\)-rational
critical point. The critical locus need not be zero-dimensional. Put
\[
B=(R/J_f)_{\mathfrak m_x},\qquad a=[f-f(x)]\in B.
\]
Then \(a\) is nilpotent and
\[
(J_f:f-f(x))_{\mathfrak m_x}
\not\subseteq
(J_f,f-f(x))_{\mathfrak m_x}
+\mathfrak m_x(J_f:f-f(x))_{\mathfrak m_x}.
\]
\end{corollary}
\begin{proof}
As above, \(d_{B/k}a=0\). Let \(\mathfrak p\) be a minimal prime of
\(B\), and let \(K=\operatorname{Frac}(B/\mathfrak p)\). The image
\(\bar a\in K\) has zero differential, so
Lemma~\ref{lem:differential-constants} shows that \(\bar a\) is
algebraic over \(k\). Let \(P(T)\in k[T]\) be its monic irreducible
polynomial. The component defined by \(\mathfrak p\) passes through the
local point \(x\), and \(a(x)=f(x)-f(x)=0\). Applying the residue map at
\(x\) to \(P(\bar a)=0\) gives \(P(0)=0\). Irreducibility then forces
\(P(T)=T\), so \(\bar a=0\). This holds for every minimal prime, hence
\[
a\in\bigcap_{\mathfrak p\in\operatorname{Min}B}\mathfrak p
=\sqrt{0_B}.
\]
The nilradical of a Noetherian ring is nilpotent, so \(a\) is
nilpotent. Apply Theorem~\ref{thm:local-strengthening}.
\end{proof}

\section{The integral-closure strengthening is false}
\label{sec:integral-counterexample}

The following example shows that the controlled term
\(\mathfrak m(J_f:f)\) in Theorem~\ref{thm:local-strengthening} cannot
be replaced by the entire integral closure.

\begin{lemma}\label{lem:filtered-regular}
Let \((S,\mathfrak m)\) be a separated \(\mathfrak m\)-adically
filtered ring, and let \(g_1,\ldots,g_r\in S\). If their initial forms
\(g_1^*,\ldots,g_r^*\) form a regular sequence in
\(\gr_{\mathfrak m}S\), then
\[
\gr_{\mathfrak m}\bigl(S/(g_1,\ldots,g_r)\bigr)
\simeq
\gr_{\mathfrak m}S/(g_1^*,\ldots,g_r^*).
\]
\end{lemma}
\begin{proof}
We argue by induction on \(r\). For \(r=1\), regularity of \(g_1^*\)
implies that for every \(h\in S\),
\[
\operatorname{ord}_{\mathfrak m}(g_1h)
=\operatorname{ord}_{\mathfrak m}(g_1)
+\operatorname{ord}_{\mathfrak m}(h),
\]
because the product of the two initial forms is nonzero. Hence the
initial ideal of \((g_1)\) is generated by \(g_1^*\), which proves the
claim. For the induction step, first pass to \(S/(g_1,\ldots,g_{r-1})\).
By the induction hypothesis its associated graded ring is the quotient
by \(g_1^*,\ldots,g_{r-1}^*\). The image of \(g_r^*\) is a
nonzerodivisor there, so the one-element argument applies once more.
\end{proof}

We apply the filtered lemma to a concrete five-variable function. The
lowest-degree parts of its Jacobian generators form a regular sequence,
so the integral closure and annihilator can first be computed in the
associated graded algebra and then transferred back to the formal
power-series ring.

\begin{theorem}\label{thm:integral-counterexample}
Let
\[
R=\mathbb Q[[x_1,\ldots,x_5]],\qquad
\ell=x_1+\cdots+x_5,
\]
and set
\[
f=\frac14\sum_{i=1}^5x_i^4+\ell^5.
\]
Then \(f\) defines an isolated hypersurface singularity and
\[
\boxed{J_f:f\subseteq\overline{J_f}.}
\]
More precisely, \(\overline{J_f}=\mathfrak m^3\).
\end{theorem}
\begin{proof}
Write
\[
g_i=\partial_i f=x_i^3+5\ell^4,\qquad
J=J_f=(g_1,\ldots,g_5).
\]
The initial forms \(x_1^3,\ldots,x_5^3\) form a regular sequence in
\(\mathbb Q[x_1,\ldots,x_5]\). Lemma~\ref{lem:filtered-regular}
therefore gives
\begin{equation}\label{eq:graded-counterexample}
\gr_{\mathfrak m}(R/J)
\simeq C:=\mathbb Q[x_1,\ldots,x_5]/(x_1^3,\ldots,x_5^3).
\end{equation}
In particular, \(J\) is \(\mathfrak m\)-primary.

Since \(J\subseteq\mathfrak m^3\) and \(\mathfrak m^3\) is integrally
closed, monotonicity of integral closure gives
\(\overline J\subseteq\mathfrak m^3\). To see the stated integral
closedness directly, use the \(\mathfrak m\)-adic order on the regular
power-series ring: an element integral over \(\mathfrak m^3\) has
order at least \(3\), by comparing the least orders in an integral
dependence equation. Every monomial of degree twelve
in five variables is divisible by some \(x_i^3\), whence
\[
\mathfrak m^{12}
=(x_1^3,\ldots,x_5^3)\mathfrak m^9
\subseteq J\mathfrak m^9+\mathfrak m^{13}.
\]
Because \(J\subseteq\mathfrak m^3\), the quotient
\(Q=\mathfrak m^{12}/J\mathfrak m^9\) is a finite \(R\)-module. The
previous inclusion says \(Q\subseteq\mathfrak mQ\), while the reverse
inclusion is automatic. Thus \(Q=\mathfrak mQ\), and Nakayama's lemma
yields
\[
\mathfrak m^{12}=J\mathfrak m^9.
\]
Thus \(J\) is a reduction of \(\mathfrak m^3\). Since
\(\mathfrak m^3\) is integrally closed in the regular local ring \(R\),
\begin{equation}\label{eq:closure-counterexample}
\overline J=\mathfrak m^3.
\end{equation}

Let \(A=R/J\) and \(a=[f]\). The identities
\[
\sum_i x_ig_i=\sum_i x_i^4+5\ell^5,\qquad
4f=\sum_i x_i^4+4\ell^5
\]
give
\begin{equation}\label{eq:a-leading}
a=-\frac14[\ell^5].
\end{equation}
In the graded algebra \(C\) of \eqref{eq:graded-counterexample},
\[
\times\ell^5:C_q\longrightarrow C_{q+5}
\]
is injective for \(q=0,1,2\). We include the representation-theoretic
verification. On \(\mathbb Q[x]/(x^3)\), define
\[
e(x^j)=x^{j+1},\qquad
h(x^j)=(2j-2)x^j,\qquad
f_-(x^j)=j(3-j)x^{j-1}.
\]
These operators satisfy the \(\mathfrak{sl}_2\) relations. On the
fivefold tensor product, the raising operator is
\[
e=e_1+\cdots+e_5=\times\ell.
\]
Every finite-dimensional \(\mathfrak{sl}_2\)-module is a direct sum of
irreducibles. Total degree \(q\) is the weight space of weight
\(\lambda=2q-10\). If an irreducible module \(V_m\) contains this
weight, then \(m\ge|\lambda|=10-2q\). The map \(e^5\) is nonzero, and
hence injective on the one-dimensional \(\lambda\)-weight space of
\(V_m\), provided \(\lambda+10\le m\). For \(q\le2\),
\[
\lambda+10=2q\le10-2q\le m.
\]
Thus \(e^5\) is injective on every irreducible summand occurring in
\(C_q\), and hence on \(C_q\) itself. Equivalently,
\[
C\simeq\bigl(\mathbb Q[x]/(x^3)\bigr)^{\otimes5},
\]
and the socle degree is \(10\).

If \(0\ne h\in\Ann_A(a)\), let \(h^*\in C_q\) be its lowest initial
form. Since \(ha=0\), the product of the initial forms of \(h\) and
\(a\) is zero in \(\gr_{\mathfrak m}A\). Equation
\eqref{eq:a-leading} therefore gives \(h^*\ell^5=0\). The
injectivity above forces \(q\ge3\), and hence
\(\Ann_A(a)\subseteq\mathfrak m_A^3\). Since
\[
(J:f)/J=\Ann_A(a),\qquad J\subseteq\mathfrak m^3,
\]
we obtain
\[
J:f\subseteq\mathfrak m^3=\overline J.
\]
\end{proof}

\begin{remark}
The example does not contradict Corollary~\ref{cor:hns}. It shows that
the right-hand side cannot be enlarged uniformly to
\(\overline{J_f}\), or to an arbitrary ideal containing
\(\mathfrak m^3\). The local strengthening uses the controlled term
\(\mathfrak m(J_f:f)\).
\end{remark}

\bibliographystyle{amsplain}
\bibliography{refs}
\end{document}